\documentclass[sn-mathphys-num]{sn-jnl}

\usepackage{graphicx}%
\usepackage{multirow}%
\usepackage{amsmath,amssymb,amsfonts}%
\usepackage{amsthm}%
\usepackage{mathrsfs}%
\usepackage[title]{appendix}%
\usepackage{xcolor}%
\usepackage{textcomp}%
\usepackage{manyfoot}%
\usepackage{booktabs}%
\usepackage{algorithm}%
\usepackage{algorithmicx}%
\usepackage{algpseudocode}%
\usepackage{listings}%

\theoremstyle{thmstyleone}%
\newtheorem{theorem}{Theorem}
\newtheorem{proposition}[theorem]{Proposition}%

\theoremstyle{thmstyletwo}%
\newtheorem{example}{Example}%
\newtheorem{remark}{Remark}%

\theoremstyle{thmstylethree}%
\newtheorem{definition}{Definition}%

\begin{document}

\title[On riemannian connections and semi-simplicity of a Lie algebra]{On riemannian connections and semi-simplicity of a Lie algebra}


\author*[1]{\fnm{Manelo} \sur{Anona}}\email{mfanona@yahoo.fr}



\affil*[1]{\orgdiv{Department of Mathematics and Computer Science}, \orgname{Faculty of Science, University of Antananarivo}, \orgaddress{, \city{Antananarivo 101}, \postcode{906}, \country{Madagascar}}}




\abstract{Using a almost product structure defined by a spray, we give a necessary and sufficient condition, for a linear connection with vanishing torsion to be Riemannian and, for the semi-simplicity of Lie algebra of projectable vector fields which commute with a spray. We show the equivalence of the semi-simplicity of a finite dimensional Lie algebra to the coincidence of the derived ideal with its algebra, to the interiority of any derivation of the algebra and, to the semi-simplicity of its adjoint representation.}




\maketitle

\section{Introduction}\label{sec1}
The object of this paper is a review and a complement of our results in \cite{ANO2}, \cite{ANO3}, \cite{ANH} and \cite{ANH1}. All considered objects are smooth. Let $M$ be a connected paracompact differentiable manifold of dimension $n\geq 2$, $J$ the vector $1-$form defining the tangent structure, $C$ the Liouville field on the tangent space $TM$, $S$ a spray. We denote $\Gamma=[J,S]$,  $\Gamma$ is an almost product structure: $\Gamma^2=I$,  $I$ being the identity vector $1$-form. We can consider $\Gamma$  \cite{GRI} as a linear connection with vanishing torsion. The curvature of $\Gamma$ is then the Nijenhuis tensor of $h$,  $R=\frac{1}{2}[h,h]$, with $h=\frac{I+\Gamma}{2}$. We will give some properties of $R$. We then study a linear connection coming from a metric. At the end, we are interested in the Lie algebra $A_S=\{X\in\chi(TM)\mid  [X,S]=0\}$, where $\chi(TM)$ denotes the set of all vector fields on $TM$.

\section{Preliminaries}\label{sec2}
We recall the bracket of two vectors $1-$form $K$ and $L$ on a manifold $M$ \cite{FN1}, 
\begin{eqnarray*}
[K,L](X,Y)&=&[KX,LY]+[LX,KY]+KL[X,Y]+LK[X,Y]-K[LX,Y]\\&&-L[KX,Y]-K[X,LY]-L[X,KY]
\end{eqnarray*}
for all $X,\ Y\in\chi(M)$.\\
The bracket $ N_L = \frac{1}{2} [L, L] $ is called the Nijenhuis tensor of $ L $. The Lie derivative $L_X$ with respect to $ X $ applied to $ L $ can be written
\begin{equation*}
[X,L]Y=[X,LY]-L[X,Y].
\end{equation*}
The exterior derivation $ d_L $ is defined in \cite{ANO1}: $ d_L = [i_L, d] $.\\
Let $ \Gamma $ be an almost product structure. We denote
\begin{eqnarray*}
	h=\frac{1}{2}(I+\Gamma)\ \text{and} \ v=\frac{1}{2}(I-\Gamma), 
\end{eqnarray*}
The vector $ 1-$form $ h $ is the horizontal projector, projector of the subspace corresponding to the eigenvalue $ + 1 $, and $ v $ the vertical projector corresponding to the eigenvalue $ -1 $. The curvature of $ \Gamma $ is defined by $ R = \frac{1}{2}[h, h] $,
which is also equal to $ \frac{1}{8}[\Gamma, \Gamma] $.\\
The Lie algebra $ A_\Gamma $ is defined by \begin{eqnarray*}
A_\Gamma=\{X\in\chi(TM)\ \text{such that}\ [X,\Gamma]=0 \}.
\end{eqnarray*}
The nullity space of the curvature $ R $ is:\begin{eqnarray*}
N_R=\{X\in \chi(TM)\ \text{ such that}\ R(X,Y)=0,\ \forall\ Y\in \chi(TM)\}.
\end{eqnarray*}
\begin{definition}
A second order differential equation on a manifold $M$ is a vector field $S$ on the tangent space $TM$ such that $JS=C$.\\
Such a vector field on $TM$ is also called a semi-spray on $M$, $S$ is a spray on $M$ if $S$ is homogeneous of degree $1$: $[C,S]=S$.\\
In what follows, we use the notation in \cite{GRI} and \cite{RUN} to express a geodesic spray of a linear connection. In local natural coordinates on an open set $U$ of $M$, $(x^i, y^j)$ are the coordinates in $TU$,  a spray $ S $ is written \begin{equation*}
S= y^i\frac{\partial}{\partial x^i}-2G^i(x^1,\ldots,x^n,y^1,\ldots,y^n)\frac{\partial}{\partial y^i}.
\end{equation*} 
\end{definition}
For a connection $\Gamma=[J,S]$, the coefficients of $ \Gamma $ become $ \Gamma^j_i = \frac{\partial G^j} {\partial y^i} $ and the projector horizontal is
\begin{eqnarray*}
		h(\frac{\partial}{\partial x^i})=\frac{\partial}{\partial x^i}-\Gamma^j_i\frac{\partial}{\partial y^j},\
		h(\frac{\partial}{\partial y^j})=0
\end{eqnarray*}
the projector vertical
	\begin{eqnarray*}
		v(\frac{\partial}{\partial x^i})=\Gamma^j_i\frac{\partial}{\partial y^j},\
		v(\frac{\partial}{\partial y^j})= \frac{\partial}{\partial y^j}
	\end{eqnarray*}
The curvature $R=\frac{1}{2}[h,h]$ become
\begin{eqnarray*}
	&& R=\frac{1}{2}R^k_{ij}dx^i\wedge dx^j\otimes \frac{\partial}{\partial y^k}\ \text{with}\ R^k_{ij}=\frac{\partial \Gamma^k_i}{\partial x^j}-\frac{\partial \Gamma^k_j}{\partial x^i}+\Gamma^l_i \frac{\partial \Gamma^k_j}{\partial y^l}-\Gamma^l_j \frac{\partial \Gamma^k_i}{\partial y^l}, \\ && i,j,k,l\in\{1,\ldots,n\}.
\end{eqnarray*}
As the functions $ G^k $ are homogeneous of degree $ 2 $, the coefficients $ \Gamma^k_{ij} = \frac{\partial^2 G^k} {\partial y^i \partial y^j} $ do not depend on $ y^i $, $ i \in \{1, \ldots, n \} $. We then have $ R^ k_{ij} = y^l R^k_ {l, ij} (x) $, the $ R^k_{l, ij} (x) $ depend only on the coordinates of the manifold $ M $.\\
\section{Properties of curvature $R$}
\begin{proposition}[\cite{RRA1}]\label{P3.1}
The horizontal nullity space of the curvature $R$ is involutive. The elements of $A_\Gamma$ are projectable vector fields.
\end{proposition}
\begin{proof}
From the expression of the curvature $R$ and taking into account $h^2=h$, we have
\begin{equation*}
R(hX,Y)=v[hX,hY],
\end{equation*}
If $hX\in N_R$, we obtain $v[hX,hY]=0$ $\forall Y\in\chi(TM)$.\\
Using the Jacobi Identity, for all $hX$ and $hY\in N_R$, we find $v[[hX,hY],hZ]=0$ $\forall Z\in\chi(TM)$. As we have $h[hX,hY]=[hX,hY]$, the horizontal nullity space of the curvature $R$ is involutive.\\
We notice that $A_\Gamma=A_h=A_v$.\\
For $X\in A_h$, we obtain
\begin{equation*}
[X,hY]=h[X,Y]\ \forall Y\in\chi(TM).
\end{equation*}
If $Y$ is a vertical vector field, we have $h[X,Y]=0$. This means that $X$ is a projectable vector field.
\end{proof}
\begin{proposition}[\cite{ANO2}]\label{P3.2}
Let $X$ be a projectable vector field. The two following relations are equivalent
\begin{enumerate}
	\item[$i)$] $[hX,J]=0$
	\item[$ii)$] $[JX,h]=0$
\end{enumerate} 
\end{proposition}
\begin{proof}
See proposition 3 of \cite{ANO2}.
\end{proof}
\begin{proposition}[\cite{ANH}]\label{P3.3}
We assume that $hN_R$ is generated as a module by projectable vector fields. If the rank of the nullity space $hN_R$ of the curvature $R$ is constant, there exists a local basis of $hN_R$ satisfying Proposition \ref{P3.2}.
\end{proposition}
\begin{proof}
See proposition 4 of \cite{ANH}.
\end{proof}
 \section{Riemannian manifolds}
Given a function $E$ from $\mathcal{T}M=TM-\{0\}$ in $\mathbb{R}^+$, with $E(0)=0$, $\mathcal{C}^\infty$ on $\mathcal{T}M$, $\mathcal{C}^2$ on the null section, homogeneous of degree two, such that $dd_JE$ has a maximal rank. The function $E$ defines a Riemannian manifold on $M$. The map $E$ is called an energy  function, its fundamental form $\Omega=dd_JE$ defines a spray $S$ by $i_Sdd_JE=-dE$ \cite{KLE}, the derivation $i_S$ being the inner product with respect to $S$. The vector $1-$form $\Gamma=[J,S]$ is called the canonical connection \cite{GRI}. The fundamental form $\Omega$ defines a metric $g$ on the vertical bundle by $g(JX,JY)=\Omega(JX,Y)$, for all $X$, $Y\in\chi(TM)$. There is \cite{GRI}, one and only one metric lift $D$ of the canonical connection such that:
\begin{eqnarray*}
&&J\mathbb{T}(hX,hY)=0,\ \mathbb{T}(JX,JY)=0\ (\mathbb{T}(X,Y)=D_XY-D_YX-[X,Y]);\\
&& DJ=0;\ DC=v; D\Gamma=0;\ Dg=0.
\end{eqnarray*}
The linear connection $D$ is called Cartan connection. We have
\begin{equation*}
D_{JX}JY=[J,JY]X,\ D_{hX}JY=[h,JY]X.
\end{equation*}
From the linear connection $D$, we associate a curvature
\begin{equation}\label{E4.1}
\mathcal{R}(X,Y)Z=D_{hX}D_{hY}JZ-D_{hY}D_{hX}JZ-D_{[hX,hY]}JZ
\end{equation}
for all $X$, $Y$, $Z\in\chi(TM)$. The relationship between the curvature $\mathcal{R}$ and $R$ is 
\begin{equation*}
\mathcal{R}(X,Y)Z=J[Z,R(X,Y)]-[JZ,R(X,Y)]+R([JZ,X],Y)+R(X,[JZ,Y]).
\end{equation*}
for all $X$, $Y$, $Z\in\chi(TM)$. In particular, 
\begin{equation*}
\mathcal{R}(X,Y)S=-R(X,Y).
\end{equation*}
In natural local coordinates on an open set $U$ of $M$, $(x^i,y^j)\in TU$, the energy function is written
\begin{equation*}
E=\frac{1}{2}g_{ij}(x^1,\ldots,x^n)y^iy^j,
\end{equation*}
where $g_{ij}(x^1,\ldots,x^n)$ are symmetric positive functions such that the matrix $(g_{ij}(x^1,\ldots,x^n))$ is invertible. And the relation $i_Sdd_JE=-dE$ gives the spray $S$
\begin{equation*}
S= y^i\frac{\partial}{\partial x^i}-2G^i(x^1,\ldots,x^n,y^1,\ldots,y^n)\frac{\partial}{\partial y^i},
\end{equation*} 
with $G_k=\frac{1}{2}y^iy^j\gamma_{ikj}$,\\
where $\gamma_{ikj}=\frac{1}{2}(\frac{\partial g_{kj}}{\partial x^i}+\frac{\partial g_{ik}}{\partial x^j}-\frac{\partial g_{ij}}{\partial x^k})$ and $\gamma_{ij}^k=g^{kl}\gamma_{ilj}$.\\
We have $G^k=\frac{1}{2}y^iy^j\gamma_{ij}^k$.
\begin{proposition}\label{P4.1}
Let $E$ be an energy function, $\Gamma$ a connection such that $\Gamma=[J,S]$. The following two relationship are equivalent:
\begin{enumerate}
	\item[$i)$] $i_Sdd_JE=-dE$;
	\item[$ii)$] $d_hE=0$.
\end{enumerate}
\end{proposition}
\begin{proof}
See proposition 1 \cite{ANH}.
\end{proof}
\begin{proposition}\label{P4.2}
For a connection satisfying the Proposition \ref{P4.1}, the scalar $1-$form $d_vE$ is completely integrable.
\end{proposition}
\begin{proof}
The Kernel of $d_vE$ is formed by vector fields belonging to the horizontal space $Im h$ ($v\circ h=0$) and vertical vector fields $JY$ such that $L_{JY}E=0$, $Y\in Im h$, taking into account $vJ=J$.\\
As we have
\begin{equation*}
[hX,hY]=h[hX,hY]+v[hX,hY]= h[hX,hY]+R(X,Y),
\end{equation*}
for all $X$, $Y\in\chi(TM)$, and that $d_hE=0$ implies $d_RE=0$. We obtain
\begin{equation*}
[hX,hY]\in Ker d_vE.
\end{equation*}
Its remains to show that $L_{v[hX,JY]}E=0$ $\forall X\in Im h$ and, $Y\in Im h$ satisfying $L_{JY}E=0$. This is immediate since we have $v=I-h$.
\end{proof}
\begin{proposition}\label{P4.3}
On a Riemannian manifold $(M,E)$, the horizontal nullity space $hN_R$ of the curvature $R$ is generated as a module by projectable vector fields belonging to $hN_R$ and, orthogonal to the image space $Im R$ of the curvature $R$ and $hN_R=hN_\mathcal{R}$.
\end{proposition}
\begin{proof}
If $R^\circ=i_SR$ is zero, then the curvature $R$ is zero; in this case, the horizontal space $Im h$ is the horizontal nullity space of $R$, isomorphic to $\chi(U)$, $U$ being a open set of $M$ \cite{ANO1}.\\
In what follows, we assume that $R^{\circ}\neq0$. According to relation (4.2) of \cite{ANH}, $JX\perp Im R\Longleftrightarrow\mathcal{R}(S,X)Y=0$ $\forall Y\in\chi(TM)$. We obtain $R(X,Y)=R^{\circ}([JY,X])$ $\forall Y\in \chi(TM)$. As $R$ is a semi-basic vector $2-$form, the above relation is only possible if $X=S$ or if $X\in hN_R$, then $X$ is generated as a module by projectable vector fields belonging to $hN_R$. We get $hN_R=hN_\mathcal{R}$.
\end{proof}
\begin{theorem}\label{T4.4}
Let $\Gamma=[J,S]$ be a linear connection. The connection $\Gamma$ comes from a energy function if and only if 
\begin{enumerate}
	\item[$(1)$] there is an energy function $E_0$ such that $d_RE_0=0$; 
	\item[$(2)$] the scalar $1-$form $d_vE_0$ is completely integrable.
\end{enumerate}
Then, there exist a constant $\varphi(x)$ on the bundle such that $e^{\varphi(x)}E_0$ is the energy function of $\Gamma$.
\end{theorem}
\begin{proof}
Both conditions are necessary according to the Proposition \ref{P4.1} and \ref{P4.2}. \\Conversely, let $E_0$ be an energy function such that $d_RE_0$. We will show that, there exist a constant $\varphi$ function on the bundle such that $d_h(e^\varphi E_0)=0$.\\
The equation is equivalent to
\begin{equation*}
d\varphi=-\frac{1}{E_0}d_hE_0.
\end{equation*}
The condition of integrability of such an equation is
\begin{eqnarray*}
d(\frac{1}{E_0})\wedge d_hE_0+\frac{1}{E_0}dd_h E_0=0,
\end{eqnarray*}
namely
\begin{equation*}
dd_hE_0=\frac{dE_0}{E_0}\wedge d_hE_0.
\end{equation*}
As $d_v E_0$ is completely integrable, we have, according to Frobenius theorem,
\begin{equation*}
dd_vE_0\wedge d_vE_0=0
\end{equation*}
Applying the inner product $i_C$ to the above equality, we get
\begin{equation*}
dd_vE_0=\frac{dE_0}{E_0}\wedge d_vE_0,
\end{equation*}
that is to say
\begin{equation*}
dd_hE_0=\frac{dE_0}{E_0}\wedge d_hE_0.
\end{equation*}
This is the condition of integrability sought.\\
For more information see \cite{ANH}.
\end{proof}
\section{Lie algebra defined by spray}
Let $A_S=\{X\in\chi(TM)\ \text{such that }[X,S]=0 \}$. By developing the calculation $[X,S]=0$, we note that the projectable elements of $A_S$ are, on an open set $U$ of $M$, of the form:
\begin{equation*}
X=X^i(x)\frac{\partial}{\partial x^i}+y^j\frac{\partial X^i}{\partial x^j}\frac{\partial}{\partial y^i}.
\end{equation*}
Denoting $\overline{\chi(M)}$ the complete lift of the vector fields $\chi(M)$ on $TM$, the projectable elements of $A_S$ are in $A_S\cap\overline{\chi(M)}$. The geodesic spray of a linear connection is defined locally by
\begin{equation*}
\ddot{x}^i=-\Gamma^i_{jk}\dot{x}^j\dot{x}^k.
\end{equation*} A result from \cite{LOO} shows that the dimension of the Lie algebra $\overline{A_S}$ is at most equal to $n^2+n$. If the dimension $\overline{A_S}$ equal to $n^2+n$, then $(M,S)$ is isomorphic to $(\mathbb{R}^n,Z_\lambda)$ for a unique $\lambda\in \mathbb{R}$, $Z_\lambda$ is given by the equations $\ddot{x}^i=\lambda \dot{x}^i$, $i=1,\dots, n$. This condition is equivalent to the nullity of the curvature $R$ of $\Gamma$ cf.\cite{ANO1}. We can see this property on example 5 of \cite{ANH1}. In the following, we are interested in the nature of the algebra $\overline{A_S}$. By associating the equality $[\overline{X},S]=0$ with the tangent structure $J$ using the Jacobi identity \cite{FN1}, we can write
\begin{equation*}
[[\overline{X},S],J]+[[S,J],\overline{X}]+[[J,\overline{X}],S]=0.
\end{equation*}
Taking into account  the hypothesis $[\overline{X},S]=0$ and a result of \cite{LEH}: $[J,\overline{X}]=0$, we find 
\begin{equation*}
[\overline{X},\Gamma]=0\ \text{with}\ \Gamma=[J,S].
\end{equation*}
We notice that $[C,J]=-J$ and $[C,\overline{X}]=0$, we then take $\Gamma=[J,S]$ with $[C,S]=S$.\\
The $1-$vector form $\Gamma$ is a linear connection without torsion in the sense of \cite{GRI}. 
\begin{proposition}[\cite{ANO2}]\label{P5.1}
The Lie algebra $\overline{A_S}$ coincides with $\overline{A_\Gamma}=A_\Gamma\cap\overline{\chi(M)}$.
\end{proposition}
\begin{proof}
See proposition 9 of \cite{ANO2}.
\end{proof} 
\begin{proposition}\cite{RRA1}\label{P5.2}
Let $H^\circ$ denote the set of projectable horizontal vector fields and $A_\Gamma\cap H^\circ=A_\Gamma^h$, then we have $A_\Gamma^h=N_R\cap H^\circ$ and $A_\Gamma^h$ is an ideal of $A_\Gamma$. 
\end{proposition}
\begin{proof}
The curvature $R$ is written, for all $X,Y\in\chi(TM)$\begin{equation*}
R(X,Y)=v[hX,hY].
\end{equation*}
If $hX\in A_\Gamma^h$, we have $R(X,Y)=v\circ h[hX,Y]=0$,$\forall Y\in\chi(TM)$. That means $X\in N_R$.\\
The curvature $R$ is written, for all $X,Y\in\chi(TM)$
 \begin{equation*}
 R(X,Y)=[hX,hY]+h^2[X,Y]-h[hX,Y]-h[X,hY].
 \end{equation*}
If $X\in N_R\cap H^\circ$, given $hX=X$ and $R(X,Y)=0$ for all $Y\in\chi(TM)$, we find $[X,hY]=h[X,hY]$.\\
If $Y$ is a vertical vector field, the above equality still holds, because it is zero.\\
For the ideal $A_\Gamma^h$, it is immediate from the expression of $A_\Gamma$.
\end{proof}
\begin{proposition}\label{P5.3}
Let $\overline{A_\Gamma}^h=A_\Gamma^h\cap\overline{A_\Gamma}$, the set of the horizontal vector fields $\overline{A_\Gamma}^h$ form a commutative ideal of $\overline{A_\Gamma}$. The dimension of $\overline{A_\Gamma}^h$ corresponds to the dimension of $A_\Gamma^h$ if the rank of $A_\Gamma^h$ is constant.
\end{proposition}
\begin{proof}
By Proposition\ref{P5.2}, $A_\Gamma^h$ is an ideal of $A_\Gamma$, so $\overline{A_\Gamma}^h=A_\Gamma^h\cap\overline{\chi(M)}$ is an ideal of $\overline{A_\Gamma}=A_\Gamma\cap\overline{\chi(M)}$, moreover $v[\overline{X},\overline{Y}]=0$, for all $\overline{X},\overline{Y}\in\overline{A_\Gamma}^h$. Propositions \ref{P3.2} and 2 of \cite{ANO2} give $J[\overline{X},\overline{Y}]=0$, for all $\overline{X},\overline{Y}\in\overline{A_\Gamma}^h$, noting that $[J,\Gamma]=0$. The horizontal and vertical parts of $[\overline{X},\overline{Y}]$ are therefore zero, that is, $[\overline{X},\overline{Y}]=0$.\\
The existence of such an element of $\overline{A_\Gamma}^h$ is given by the proposition\ref{P3.3}.
\end{proof}
\section{Case of constant values of $\overline{A_\Gamma}$}
If we expand the equation $[\overline{X},S]=0$ with $S=[C,S]$, we get
\begin{equation*}
X^l\frac{\partial \Gamma^k_{ij}}{\partial x^l}+\frac{\partial X^l}{\partial x^j}\Gamma^k_{il}+\frac{\partial X^l}{\partial x^i}\Gamma^k_{lj}+\frac{\partial^2 X^k}{\partial x^i\partial x^j}-\frac{\partial X^k}{\partial x^l}\Gamma^l_{ij}=0.
\end{equation*}
we note that the constant  values of $\overline{A_\Gamma}$ verify 
\begin{equation}\label{E6.1}
X^l\frac{\partial \Gamma^k_{ij}}{\partial x^l}=0
\end{equation}
\begin{proposition}\label{P6.1}
Le $\Gamma$ be a linear connection without torsion. If the constant vector fields of $\overline{A_\Gamma}$ form a commutative ideal of $\overline{A_\Gamma}$, they are at most the constant elements of an ideal $I$ of affine vector fields containing these constants such that for all $\overline{X}\in\overline{A_\Gamma}$, $\overline{X}$ is written $\overline{X}=\overline{X_1}+\overline{X_2}$ with $\overline{X_2}\in I$ and that $[\overline{X_1},\overline{X_2}]=0$, the derived ideal of $\overline{A_\Gamma}$ never coincides with $\overline{A_\Gamma}$.
\end{proposition}
\begin{proof}
\begin{description}
	\item[1st case:] The functions $G^k$ do not depend on some coordinates in an open set $U$ of $M$. To simplify, quite to change the numbering order of the coordinates, the spray $S$ is such that $\frac{\partial G^k}{\partial x^{p+1}}=0$, $\dots$, $\frac{\partial G^k}{\partial x^{n}}=0$, $k\in\{1,\dots, n\}$ and $1\leq p\leq n-1$. Then, we have $\frac{\partial}{\partial x^{p+1}},\dots, \frac{\partial}{\partial x^{n}}\in\overline{A_\Gamma}(U)$.\\
	For any $\overline{X}\in \overline{A_\Gamma}(U)$, we can write
	\begin{eqnarray*}
\overline{X}&=& X^i\frac{\partial}{\partial x^i}+y^j\frac{\partial X^i}{\partial x^j}\frac{\partial}{\partial y^i}\\
&=& \sum_{l=1}^{p}(X^l\frac{\partial}{\partial x^l}+y^j\frac{\partial X^l}{\partial x^j}\frac{\partial}{\partial y^l})+\sum_{r=p+1}^{n}(X^r\frac{\partial}{\partial x^r}+y^j\frac{\partial X^r}{\partial x^j}\frac{\partial}{\partial y^r}),\ 1\leq j\leq n.
\end{eqnarray*}
For the Lie sub-algebra generated by $\{\frac{\partial}{\partial x^{p+1}},\dots, \frac{\partial}{\partial x^{n}}\}$ form an ideal of $\overline{A_\Gamma}(U)$, we must have $[\frac{\partial}{\partial x^{h}},\overline{X}]$ belong to this ideal for all $h$, $p+1\leq h\leq n$.\\
That implies $\frac{\partial X^l}{\partial x^{h}}=0$, for all $l$ such that $1\leq l\leq p$ and for all $h$ such that $p+1\leq h\leq n$.\\
We have $X^r=a_s^r x^s+b^r$, $p+1\leq r,s\leq n$; $a_s^r,b^r\in \mathbb{R}$.\\
Denoting 
\begin{eqnarray*}
\overline{X_1}&=&\sum_{l=1}^{p}(X^l\frac{\partial}{\partial x^l}+y^j\frac{\partial X^l}{\partial x^j}\frac{\partial}{\partial y^l}),\ 1\leq j\leq n\\
\overline{X_2}&=&\sum_{r=p+1}^{n}(a_s^r x^s+b^r)\frac{\partial}{\partial x^r}+a_s^r y^s\frac{\partial}{\partial y^r},\ p+1\leq s\leq n.
\end{eqnarray*}
An element $\overline{X}\in\overline{A_\Gamma}$, $\overline{X}$ is written $\overline{X}=\overline{X_1}+\overline{X_2}$ with $[\overline{X_1},\overline{X_2}]=0$.
\item[2nd case:] The elements of $\overline{A_\Gamma}$ are of the form $a^l\frac{\partial}{\partial x^l}$, $l\in\{1,\dots, p\}$. The decomposition of the elements of $\overline{A_\Gamma}$ amounts  to the same way. In any case, the derived ideal of $\overline{A_\Gamma}$ never coincides with $\overline{A_\Gamma}$.
\end{description}
\end{proof}
\begin{theorem}\label{T6.2}
The Lie algebra $\overline{A_\Gamma}$ is semi-simple if and only if the horizontal and projectable vector fields of the nullity space of the curvature $R$ is zero and the derived ideal of $\overline{A_\Gamma}$ coincides with $\overline{A_\Gamma}$.
\end{theorem}
\begin{proof}
If the Lie algebra $\overline{A_\Gamma}$ is semi-simple, any commutative ideal of $\overline{A_\Gamma}$ reduces to zero by definition. According to the proposition \ref{P5.3}, the horizontal and projectable vector fields of the nullity space of the curvature $R$ of $\Gamma$ is zero. The derived ideal of $\overline{A_\Gamma}$ coincides with $\overline{A_\Gamma}$ by a classical result.\\
Conversely, if $\overline{X}\in\overline{A_\Gamma}$, we have $[\overline{X},h]=0$. According to the Jacobi Identity cf.\cite{FN1} $[\overline{X},[h,h]]=0$, ie. $[\overline{X},R]=0$. We can write $[\overline{X},R(Y,Z)]=R([\overline{X},Y],Z)+R(Y,[\overline{X},Z])$, for all $Y,Z\in\chi(TM)$. If $\overline{X}$ and $\overline{Y}$ are elements of a commutative ideal of $\overline{A_\Gamma}$, we find
\begin{eqnarray}\label{E6.2}
[\overline{X},R(\overline{Y},Z)]=R(\overline{Y},[\overline{X},Z]),\ \forall Z\in\chi(TM).
\end{eqnarray}
If the horizontal and projectable vector fields of the nullity space of the curvature $R$ is zero, the semi-basic vector $2-$form $R$ is non-degenerate on $\overline{\chi(M)}\times\chi(TM)$. The only possible case for the equations (\ref{E6.2}) is that the commutative ideal of $\overline{A_\Gamma}$ is at most formed by constant vector fields of $\overline{A_\Gamma}$, according to the proposition \ref{P6.1}, the derived ideal of $\overline{A_\Gamma}$ never coincides with $\overline{A_\Gamma}$ if this ideal formed by constant vector fields is not zero.
\end{proof}

\begin{example}
We take $ M = \mathbb{R}^3 $, a spray $ S $:
\begin{equation*}
S= y^1\frac{\partial}{\partial x^1}+y^2\frac{\partial}{\partial x^2}+y^3\frac{\partial}{\partial x^3} -2(e^{x^3}(y^1)^2+y^2y^3) \frac{\partial}{\partial y^1}.
\end{equation*}
and the linear connection $\Gamma=[J,S]$. The non-zero coefficients of $ \Gamma $ are
\begin{eqnarray*}
&&\Gamma^1_1=2e^{x^3}y^1,\ \Gamma^1_2= y^3,\ \Gamma^1_3=y^2.
\end{eqnarray*}
A base of the horizontal space of $\Gamma$ is written
\begin{eqnarray*}
&&\frac{\partial}{\partial x^1}-2e^{x^3}y^1\frac{\partial}{\partial y^1},\\
&&\frac{\partial}{\partial x^2}-y^3\frac{\partial}{\partial y^1},\\
&&\frac{\partial}{\partial x^3}-y^2\frac{\partial}{\partial y^1}.
\end{eqnarray*}
The horizontal nullity space of the curvature is generated as a module by 
\begin{eqnarray*}
(y^1-y^2)\frac{\partial}{\partial x^2}+y^3\frac{\partial}{\partial x^3}-y^1y^3\frac{\partial}{\partial y^1}.
\end{eqnarray*}
The horizontal nullity space is not generated as a module by projectable vector fields in $hN_R$. This linear connection according to the proposition \ref{P4.3} cannot come from an energy function.\\

The Lie  algebra $ \overline {A_\Gamma} $ is generated as Lie algebra by:
\begin{eqnarray*}
g_1= x^1\frac{\partial}{\partial x^1}+x^2\frac{\partial}{\partial x^2}-\frac{\partial}{\partial x^3}+y^1\frac{\partial}{\partial y^1}+y^2\frac{\partial}{\partial y^2},\ g_2=\frac{\partial}{\partial x^1},\
g_3=\frac{\partial}{\partial x^2}.
\end{eqnarray*}
The Lie algebra $\overline{A_\Gamma}$ is that of affine vector fieds containing the commutative ideal $\{g_2,g_3\}$.
\end{example}
\section{Lie algebras of infinitesimal isometries}
\begin{definition}
A vector field $X$ on a Riemannian manifold $(M,E)$ is called infinitesimal automorphism of  the symplectic form $\Omega$ if $L_X\Omega=0$.\\
The set of infinitesimal automorphisms of $\Omega$ forms a Lie algebra. We denote this Lie algebra by $A_g$, in general of infinite dimension.
\end{definition}
\begin{theorem}\label{T7.2}
We denote $\overline{A_g}=A_g\cap \overline{\chi(M)}$. The Lie algebra $\overline{A_g}$ is semi-simple if and only if the horizontal nullity space of the Nijenhuis tensor of $\Gamma$ is zero and, the derived ideal of $\overline{A_g}$ coincides with $\overline{A_g}$.
\end{theorem}
\begin{proof}
This is the application of proposition \ref{P4.3} and theorem \ref{T6.2}.\\
For more information, see \cite{ANO3} and \cite{ANH1}.
\end{proof}
\begin{example}
We take $ M = \mathbb{R}^4 $ and the energy function is written:
$$E=\frac{1}{2}(e^{x^3}(y^1)^2+(y^2)^2+e^{x^1}(y^3)^2+e^{x^2}(y^4)^2).$$
The non-zero linear connection coefficients are
\begin{eqnarray*}
&&\Gamma^1_1=\frac{y^3}{2},\ \Gamma^1_3= -\frac{y^3e^{x^1-x^3}-y^1}{2},\ \Gamma^2_4=-\frac{y^4e^{x^2}}{2},\\  &&\Gamma^3_1=-\frac{y^1e^{x^3-x^1}-y^3}{2},\ \Gamma^3_3=\frac{y^1}{2},\ \Gamma^4_2= \frac{y^4}{2},\ \Gamma^4_4= \frac{y^2}{2}.
\end{eqnarray*}
The horizontal nullity space of the curvature is zero. \\
The Lie algebra $ \overline {A_\Gamma} $ is generated as Lie algebra by:
\begin{eqnarray*}
g_1&=& x^4 \frac{\partial}{\partial x^2}-(-e^{-x^2}+\frac{(x^4)^2}{4})\frac{\partial}{\partial x^4}+y^4\frac{\partial}{\partial y^2}-(\frac{x^4y^4}{2}+y^2 e^{-x^2})\frac{\partial}{\partial y^4},\\
g_2&=& -2\frac{\partial}{\partial x^2}+x^4\frac{\partial}{\partial x^4}+y^4\frac{\partial}{\partial y^4},\
g_3=\frac{\partial}{\partial x^4},\
g_4= \frac{\partial}{\partial x^1}+\frac{\partial}{\partial x^3}.
\end{eqnarray*}
We see that $g_4$ is the center of $\overline{A_\Gamma}$ corresponding to the second case of the proposition \ref{P6.1}, while the Lie algebra $\overline{A_g}$ is generated as a Lie algebra by $g_1,\ g_2,\ g_3$. The  Lie algebra $\overline{A_g}$ is simple and isomorphic to $sl(2)$.
\end{example}
\section{Finite dimensional Lie algebra}
In this section, we consider only a finite-dimensional Lie algebra over the field $\mathbb{K}$ of zero characteristic class. The notions and notations are those of \cite{BOU}.
\begin{theorem}
The Lie algebra $\mathfrak{g}$ over $\mathbb{K}$ of finite dimension is semi-simple if and only if the derived ideal of $\mathfrak{g}$ coincides with $\mathfrak{g}$, any derivation of $\mathfrak{g}$ is inner and, its adjoint representation is semi-simple.
\end{theorem}
\begin{proof}
If $\mathfrak{g}$ is a semi-simple Lie algebra over $\mathbb{K}$ of finite dimension, the three conditions are verified \cite{BOU}.\\
Conversely, if the adjoint representation of $\mathfrak{g}$ is semi-simple, $\mathfrak{g}$ is a product of a semi-simple Lie algebra and an abelian Lie algebra cf.\cite{BOU}. Let be $I$ the abelian Lie algebra of $\mathfrak{g}$. For $a,b\in \mathfrak{g}$ such that $a,b\notin I$ and $[a,b]\neq 0$, then we have $[a,b]\notin I$. Let $e_1,\dots,e_p,e_{p+1},\dots, e_n$ be a base of $\mathfrak{g}$ such that $e_1,\dots,e_p$ $(p<n)$ belong to $I$. By defining a linear map $D$ such that $D(e_i)=e_i$, $1\leq i\leq p$, and $D(e_{p+1})=0,\dots, D(e_n)=0$, the map $D$ is a derivation of $\mathfrak{g}$, its trace function is equal to $p$. If $\mathfrak{g}=[\mathfrak{g},\mathfrak{g}]$, the adjoint representation of $\mathfrak{g}$ belongs to $sl(\mathfrak{g})$ which is semi-simple \cite{BOU} p.71. Its trace function is zero. We end up with a contradiction if $I\neq 0$ and if the derivation $D$ is inner.
\end{proof}
\begin{remark}\label{R.8.3}
We can see such reasoning in an example \cite{ANH1} \S 5.
\end{remark}
\begin{remark}\label{R.8.2}
 For a Lie algebra of countable dimension, see our results in \cite{RRA2}.
\end{remark}

The author has never worked with other entities apart from the Malagasy students he directed for research. The author became aware of these problems in the 1980s when he studied and researched at the Institut Fourier Grenoble I under the direction of Professor Joseph Klein.

\end{document}